\RequirePackage{fix-cm}
\documentclass[smallextended]{svjour3}       
\smartqed  
\usepackage{graphicx}
\usepackage{amssymb,amsfonts,dsfont,color}
 \usepackage{amsmath}
\begin{document}

\title{Noncoercive and Noncontinuous Equilibrium Problems 
}
\subtitle{Existence Theorem in Infinite Dimensional Spaces}


\author{F. Fakhar \and H. R. Hajisharifi \and Z. Soltani  
}


\institute{F. Fakhar \at
              Department of Pure Mathematics, Faculty of Mathematics and Statistics, University of Isfahan, Isfahan 81746-73441, Iran\\
              \email{f.fakhar66@gmail.com}           
           \and
           H. R. Hajisharifi \at
            Department of Mathematics, Khansar campus, University of Isfahan, Isfahan 81746-73441, Iran\\
            School of Mathematics, Institute for Research in Fundamental Sciences (IPM), P.O. Box: 19395-5746, Tehran, Iran\\
            \email{h.r.hajisharifi@sci.ui.ac.ir}
            \and
           Z. Soltani (corresponding author) \at
            Department of Pure Mathematics, University of Kashan, Kashan, 87317-53153,  I. R. Iran\\
            \email{z.soltani@kashanu.ac.ir}
}

\date{Received: date / Accepted: date}

\maketitle

\begin{abstract}
In this paper, we extend the definition of $qx$-asymptotic function, for extended real-valued function that define on an infinite dimensional topological normed space without lower semicontinuity or quasi-convexity condition. As the main result, by using some asymptotic conditions, we obtain sufficient optimality conditions for existence of solutions to equilibrium problems, under weaker assumptions of continuity and convexity, when the feasible set is an unbounded subset of an infinite dimensional space. Also, as a corollary, we obtain a necessary and sufficient optimality conditions for existence of solutions to equilibrium problems with unbounded feasible set.\\
Finally, as an application, we establish a result for existence of solutions to minimization problems.

\end{abstract}

\section{Introduction}
\label{intro}
Let $K$ be a nonempty subset of topological norm space $X$ and let $\psi:K\times K\rightarrow \Bbb{R}$ be a real-valued bifunction. The equilibrium problem was firstly formulated as follows by Blum and Ottli \cite{Blum.Oet}:
$$(EP)\quad \mbox{find}\quad \bar{x}\in K\quad\mbox{such that}\quad \psi(\bar{x},y)\geq 0, \quad\forall y\in K.$$
The set of its solutions is denoted by $S(\psi,K)$.
Equilibrium problems were first introduced in Ky Fan's pioneering work \cite{KyFan}, and they have been explored extensively over the years.
The equilibrium problems were studied in mathematics and economics; for example, see \cite{AWY,ACI,CS,DM,FB2,IL3,IL4,IS,Ka,Lara,LA,LAYL,Nash,NFr,Re} and the references therein.
Minimization problems, mixed variational problems, vector optimization problems, and linear complementary problems are all examples of classical mathematics problems that can be formulated as equilibrium problems.\\
If $K$ is a nonempty, compact, and convex set and $\psi$ is upper semicontinuous in its first argument and quasi-convex in its second argument, then $(EP)$ has a solution \cite{Oettli}. The study existence of solutions to $(EP)$ has received a lot of attention in recent years, and various existence results have been obtained where $K$ is compact or $\psi$ is under a coercive condition. We have the following result of \cite{KQ} as a recent result in this direction.
\begin{theorem}\label{main thm1}
	Let $Z$ be a compact topological space and $\psi:Z\times Z\rightarrow \Bbb{R}$ be cyclically anti-quasimonotone such that for each $y\in Z$, the set $[\psi(.,y)\geq 0]:=\{x\in Z:\psi(x,y)\geq 0\}$ is closed. Then, there exists $\bar{x}\in Z$ such that $\psi(\bar{x},y)\geq 0$, for all $y\in Z$.
\end{theorem}
We should note that, when the feasible set is compact or some coercivity assumptions are imposed on the function, it is rather easy to establish an optimality condition for $(EP)$. Then many researchers have studied existence of solutions to $(EP)$, by weakening the compactness and convexity conditions on $\psi$ and the coerciveness condition on $\psi$ until now; for example, see \cite{ACR,FB,IKS,KL,NT} and the references therein. Also, in \cite{NT} the authors used the new concept of generalized continuity, so-called transfer semicontinuity, to reduce the condition of upper semicontinuity of $\psi$ in the $(EP)$.\\
The so-called asymptotic directions and functions, which gave rise to the field of mathematics known as asymptotic analysis, are one technique for dealing with unbounded data. In \cite{FB}, the author, motivated by \cite{AGT,BBGT}, treated the noncoerciveness condition on $\psi$ by using the asymptotic analysis to the study of $(EP)$. Motivated by \cite{FB}, the notion of asymptotic function will be used to treat the noncoerciveness condition in this work.

This paper is organized as follows. In section 2, we have some basic definitions and results. In section 3, we extend the definition of $qx$-asymptotic function \cite{H1}, for extended real-valued function that define on an infinite dimensional space. Also, we establish some property of this asymptotic function. In section 4, as the main result of this paper, we verify new sufficient optimality conditions for existence of an $(EP)$ solution, when $K$ is an unbounded set, under weaker assumptions of continuity and generalized convexity, via asymptotic analysis. Also, in this section, we obtain a necessary and sufficient optimality condition for existence of solutions to $(EP)$. Finally, in section 5, as an application of our main result, we use asymptotic conditions to investigate existence of solutions to minimum problems, when the feasible set is an unbounded set.
\section{Preliminaries}
\label{sec:1}
In this section, we give some notions and results that will be used in this paper. The set of real numbers (resp. natural numbers and rational numbers) is denoted by $\Bbb R$ (resp. $\Bbb{N}$ and $\Bbb Q$). Throughout the paper suppose that $Y$ is a nonempty set, $(X,\|.\|)$ is a real normed space and $\sigma$ is a Hausdorff topology coarser than the strong one on $X$. As \cite{GLN}, when necessary we shall consider the following condition on $(X,\sigma)$.

{\bf Assumption} ${\bf(H_{\sigma})}$ The closed unit ball $B_X$ in $X$ is sequentially $\sigma$-compact, that is, $B_X$ is sequentially compact with $\sigma$-topology.

For instance, when $X$ is a reflexive Banach space and $\sigma$ is the weak topology on $X$ or $X$ is the dual of a separable Banach space and $\sigma$ is the weak$^{*}$ topology on $X$, the assumption $(H_\sigma)$ holds, \cite{GLN}.

We assume that $K$ is a nonempty subset of $(X,\|.\|)$ and for each $n\in\Bbb{N}$, $K_n:=\{x\in K: \|x\|\leq n\}$.
The convergence of a sequence in $X$ with $\sigma$-topology, will be denoted by $x_n\xrightarrow{\sigma}x$. Also, the convex hull, closure and $\sigma$-closure of $K$ are denoted by $\mathrm{co}(K)$, $\mathrm{cl}(K)$ and $\mathrm{cl}^\sigma(K)$, respectively.\\
Let $f:K\rightarrow \Bbb{R}\cup \{+\infty\}$ be a function. The effective domain of $f$ is defined by $dom (f):=\{x\in K: f(x)<+\infty\}$, moreover we say that $f$ is a proper function when $dom(f)$ is nonempty. We denote the sublevel set of $f$ at $\lambda\in \Bbb{R}$ by $[f\leq\lambda]:=\{x\in K: f(x)\leq\lambda\}$. Let $dom(f)$ be convex, then $f$ is said to be quasi-convex, if for every $x,y\in dom (f)$ and $\alpha\in [0,1]$,
$$f(\alpha x+(1-\alpha)y)\leq \max\{f(x),f(y)\}.$$
It is well-known that $f$ is quasi-convex iff $[f\leq\lambda]$ is convex, for all $\lambda\in \Bbb{R}$.
Recall that every convex function is quasi-convex. Also $f$ is said to be quasi-concave, if $-f$ be a quasi-convex function. For a further study on generalized convexity; see \cite{C1,H2}.\\
The function $f:K\rightarrow \Bbb{R}\cup \{+\infty\}$ is called \cite{T1}:
\begin{itemize}
	\item[$\bullet$] $\sigma$-lower semicontinuous ($\sigma$-lsc from now on) at $x_0\in K$ if, $[f\leq\lambda]$ is a $\sigma$-closed set, for all $\lambda\in \Bbb{R}$.
  \item[$\bullet$] sequentially $\sigma$-lower semicontinuous (sequentially $\sigma$-lsc from now on) at $x_0\in K$ if, $[f\leq\lambda]$ is a sequentially $\sigma$-closed (i.e., sequentially closed with $\sigma$-topology) set, for all $\lambda\in \Bbb{R}$.
	\item[$\bullet$] $\sigma$-upper semicontinuous ($\sigma$-usc from now on) at $x_0\in K$ if, $-f$ is $\sigma$-lsc at $x_0$.
	\item[$\bullet$] $\sigma$-transfer lower continuous at $x_0\in K$ if, $f(x_0)>f(y)$ for some $y\in K$, implies that there exist $x'\in K$ and a $\sigma$-neighborhood $\mathcal{V}_{\sigma}(x_0)\subseteq K$ of $x_0$ such that, $f(x)> f(x')$ for all $x\in \mathcal{V}_{\sigma}(x_0)$.
	\item[$\bullet$] $\sigma$-transfer lower continuous (resp. $\sigma$-lsc, sequentially $\sigma$-lsc, $\sigma$-usc) if, $f$ is $\sigma$-transfer lower continuous (resp. $\sigma$-lsc, sequentially $\sigma$-lsc, $\sigma$-usc) at each $x\in K$.
\end{itemize}
For a further study on generalized continuity of extended real-valued function; see \cite{A1} and the references therein.\\
\begin{definition}\cite{IL,NT}\label{Def tlc}
The bifunction $\psi: K\times Y\rightarrow \Bbb{R}$ is said to be
	\begin{itemize}
		\item[$\bullet$] $\sigma$-transfer lower semicontinuous ($\sigma$-tlsc from now on) in $x_0$ with respect to $Y$ if, for $(x_0,y)\in K\times Y$, $\psi(x_0,y)>0$ implies that there exist some point $y'\in Y$ and some $\sigma$-neighborhood $\mathcal{V}_{\sigma}(x_0)\subseteq K$ of $x_0$ such that $\psi(x, y')>0$ for all $x\in\mathcal{V}_{\sigma}(x_0)$. Also, we say that $\psi$ is $\sigma$-tlsc on $K$, when $Y=K$ and $\psi$ is $\sigma$-tlsc in $x$ with respect to $Y$ for every $x\in K$.
		\item[$\bullet$] locally dominated in $y$ if, for any finite subset $\{y_1,y_2,...,y_n\} \subseteq Y$, there exists $x\in K$ such that
		$$\max_{i\in\{1,2,...,n\}}\psi(x,y_i)\leq 0.$$
		\item[$\bullet$] transfer quasi-concave in $y$ if, for any finite subset $\{y_1,y_2,...,y_n\} \subseteq Y$, there exists a corresponding finite subset $\{x_1,x_2,...,x_n\} \subseteq K$ such that, for any subset $L\subseteq \{1,2,...,n\}$ and any $x\in \mathrm{co}\{x_i:i\in L\}$, we have $\min_{i\in L}\psi(x,y_i)\leq 0$.\\
		Also the bifunction $\psi: K\times K\rightarrow \Bbb{R}$ is said to be
		\item[$\bullet$] pseudomonotone on $K$ if, for every $x,y\in K$;
		$$\psi(x,y)\geq 0\Rightarrow \psi(y,x)\leq 0.$$
		\item[$\bullet$] cyclically anti-quasimonotone if, for any points $x_1,\ldots,x_n\in K$, possibly not all different, there exists an $i\in\{1,\ldots,n\}$ such that $\psi(x_i,x_{i+1})\geq 0$, where $x_{n+1}:=x_1$.
	\end{itemize}
\end{definition}
Motivated by Definition \ref{Def tlc}, the bifunction $\psi: K\times Y\rightarrow \Bbb{R}$ is said to be $\sigma$-transfer upper semicontinuous ($\sigma$-tusc from now on) in $x_0\in K$ with respect to $Y$, when $-\psi$ be $\sigma$-tlsc in $x_0$ with respect to $Y$, and $\psi$ is said to be $\sigma$-tusc on $K$, when $-\psi$ is $\sigma$-tusc on $K$.
Moreover, we say that $\psi$ is anti-locally dominated in $y$ (transfer quasi-convex in $y$) when $-\psi$ is  locally dominated in $y$ (transfer quasi-concave in $y$).
\begin{remark}\label{Remark c.a.qm}
By \cite{KQ}, let $\psi: K\times K\rightarrow \Bbb{R}$ is a cyclically anti-quasimonotone bifunction, then
	\begin{itemize}
		\item[(i)] for any $A\subseteq K$, $\psi|_A: A\times A\rightarrow \Bbb{R}$ is cyclically anti-quasimonotone.
		\item[(ii)] if $\psi': K\times K\rightarrow \Bbb{R}$ is defined by $\psi'(x,y)=\psi(y,x)$ for all $x,y\in K$, then $\psi'$ is cyclically anti-quasimonotone, too.
\end{itemize}
\end{remark}
In the following definition, we recall some generalizations of set-valued mappings with closed valued.
\begin{definition}\cite{LSS}
Let $\Lambda$ be a nonempty set. The set valued mapping $G:\Lambda\rightrightarrows X$ is said to be
\begin{itemize}
	 \item[(i)] $\sigma$-transfer closed on $\Lambda$, if
	$$\bigcap_{x\in \Lambda}\mathrm{cl^\sigma}(G(x))=\bigcap_{x\in \Lambda}G(x).$$
	\item[(ii)] $\sigma$-intersectionally closed on $\Lambda$, if
	$$\bigcap_{x\in \Lambda}\mathrm{cl^\sigma}(G(x))=\mathrm{cl^\sigma}(\bigcap_{x\in \Lambda}G(x)).$$
\end{itemize}
\end{definition}
In the sequel we have some basic definitions and properties of the $\sigma$-asymptotic cone and the $\sigma$-asymptotic function.

The $\sigma$-asymptotic cone $K^\infty_{\sigma}$ of $K$ is defined by \cite{A3},
$$K^\infty_{\sigma}:=\{u\in X: \exists t_k\rightarrow +\infty~\mbox{and}~ (x_k)\subseteq K~\mbox{s.t.}~ \frac{x_k}{t_k}\xrightarrow{\sigma} u\},$$
with the convention $\emptyset^\infty_{\sigma}=\emptyset$.
\begin{proposition} \cite{A3,GLN} \label{Prop1}
Let $A$ and $B$ be two nonempty subsets of $X$. Then the following hold.
	\begin{itemize}
		\item[(i)] Let $A\subseteq B$, then $A^\infty_{\sigma}\subseteq B^\infty_{\sigma}$.
		\item[(ii)] Let $A$ be a sequentially $\sigma$-closed and convex set, then $A^\infty_{\sigma}$ is sequentially $\sigma$-closed and convex, and for each $x\in A$ it follows that
	$$A^\infty_{\sigma}=\{u\in X: x+tu\in A,~\forall t>0\}.$$
		 \item[(iii)] Let $A$ be a sequentially $\sigma$-closed and convex set, then $A+ A^\infty_{\sigma}\subseteq A.$
		\item[(iv)] Let $(K_i)_{i\in I}$ be a family of nonempty sets in $X$, then
		$$(\bigcap_{i\in I}K_i)^\infty_{\sigma}\subseteq \bigcap_{i\in I}(K_i)^\infty_{\sigma}.$$
		In addition, the equality holds when $K_i$ is sequentially $\sigma$-closed and convex for each $i\in I$, and $\cap_{i\in I}K_i\neq\emptyset$.
	\end{itemize}
\begin{remark}\label{Rem1.cc}
Let $A$ be a nonempty $\sigma$-closed and convex subset of $X$. Clearly, for $x\in A$, the set $\{u\in X: x+tu\in A,~\forall t>0\}$ is $\sigma$-closed and convex. Then by Proposition \ref{Prop1}(ii), $A^\infty_{\sigma}$ is a $\sigma$-closed and convex cone.
\end{remark}

\end{proposition}
For any given proper function $f:X\rightarrow \Bbb{R}\cup\{+\infty\}$, the $\sigma$-asymptotic function $f^\infty_{\sigma}:X\rightarrow \Bbb{R}\cup\{\pm\infty\}$ is defined by \cite{GLN},
$$f^\infty_{\sigma}(u)=\inf\{\liminf_{k\rightarrow +\infty}\frac{f(t_ku_k)}{t_k}: t_k\rightarrow +\infty, u_k\xrightarrow{\sigma} u\}.$$
Note that, when $f:K\rightarrow \Bbb{R} \cup\{+\infty\}$ be a proper function, we extend this function to all of $X$ by setting $f(x)=+\infty$, if $x\in X\setminus K$, and we obtain $f^\infty_{\sigma}$.

\section{A Generalized Asymptotic Function}
\label{sec:2}
At first, we extend the definition of $qx$-asymptotic function \cite{H1}, for extended real-valued function that define on an infinite dimensional topological normed space and it has neither lower semicontinuity nor quasi-convexity conditions. Also, we establish some property of this asymptoptic function.
\begin{definition}
	Let $f:X\rightarrow \Bbb{R}\cup \{+\infty\}$ be a proper function. The $\sigma$-generalized asymptotic ($\sigma g$-asymptotic) function  $f^{\sigma g}: X\rightarrow \Bbb{R}\cup \{\pm\infty\}$ is defined by
	$$f^{\sigma g}(u) :=\inf\{\lambda: u\in ([f\leq \lambda])^\infty_{\sigma}\}.$$
\end{definition}
For the proper function $f:K\rightarrow \Bbb{R}\cup\{+\infty\}$, we extend this function to all of $X$ by setting $f(x)=+\infty$, if $x\in X\setminus K$, and we obtain the $\sigma g$-asymptotic function $f^{\sigma g}$. It is clear that, if $f: \Bbb{R}^n \to \Bbb{R}\cup \{+\infty\}$ is a proper lsc (i.e. $\sigma$-lsc with $\sigma=\|.\|$) and quasi-convex function, then $f^{\sigma g}=f^{qx}$.
Notice that, there is no connection between $f^{\sigma g}$ and $f^\infty_\sigma$ or $f^\infty_q$, that was presented in \cite{IL4}. Indeed, similar to example in \cite{H1}, we consider the constant real-valued function $f(x)=\alpha$, for all $x\in X$.
We have $f^{\sigma g}\equiv \alpha$ and $f^\infty_q=f^\infty_\sigma\equiv 0$.
Hence, for $\alpha > 0$ we have $f^{\sigma g}> f^\infty_q=f^\infty_\sigma$, while for $\alpha < 0$ we have $f^\infty_\sigma=f^\infty_q > f^{\sigma g}$.\\
Let $f$ be quasi-convex and sequentially $\sigma$-lsc, thus similar to equation (12) in \cite{H1}, by Proposition \ref{Prop1}(iv) for each $\lambda\in \Bbb{R}$ with $[f\leq\lambda]\neq\emptyset$ we have
$$[f^{\sigma g}\leq\lambda]=\bigcap_{\mu>\lambda}([f\leq\mu])^{\infty}_{\sigma}=(\bigcap_{\mu>\lambda}[f\leq\mu])^{\infty}_{\sigma}=([f\leq\lambda])^{\infty}_{\sigma}.$$
Then, in this case, by $\emptyset^\infty_{\sigma}=\emptyset$ for each $\lambda\in \Bbb{R}$ we have
\begin{equation}\label{Eq1}
[f^{\sigma g}\leq\lambda]=([f\leq\lambda])^{\infty}_{\sigma}.
\end{equation}
Also, by the similar proof to \cite[Remark 3.1(i)]{H1} we have, $f^{\sigma g}$ is order preserving in the sense that, for two proper functions $f_1,f_2:X\rightarrow \Bbb{R}\cup\{+\infty\}$,
		\begin{center}
			$f_1\leq f_2$ implies that $(f_1)^{\sigma g}\leq (f_2)^{\sigma g}$.
		\end{center}

In the following, an analytic formula for the $\sigma g$-asymptotic function is given.
\begin{proposition}\label{propo FFH}
	Let $f:X\rightarrow \Bbb{R}\cup\{+\infty\}$ be a proper function, then for each $u\in X$ we have
	$$f^{\sigma g}(u)=\inf\{\liminf_{n\to +\infty} f(t_n d_n): t_n\rightarrow +\infty, d_n\xrightarrow{\sigma} u\}.$$
\end{proposition}
\begin{proof}
	Let $u\in X$ and $\alpha_u=\inf\{\liminf_{n\to +\infty} f(t_n d_n): t_n\rightarrow +\infty,d_n\xrightarrow{\sigma} u\}$.\\
If $\alpha_u<\mu$, there exist $t_n\rightarrow +\infty$ and $d_n\xrightarrow{\sigma} u$ such that $\lim_{n\to +\infty} f(t_n d_n)<\mu$. Thus, for sufficiently large $n\in \Bbb{N}$, $f(t_n d_n)<\mu$ and so $t_n d_n\in [f\leq \mu]$. Hence, $u\in ([f\leq \mu])^\infty_{\sigma}$, and therefore $f^{\sigma g}(u)\leq \mu$.
	Then
	$$f^{\sigma g}(u)\leq \alpha_u.$$
	For the converse, suppose that $f^{\sigma g}(u)<\mu$. Hence $u\in ([f\leq \mu])^\infty_{\sigma}$. Then there exist $(x_n)\subseteq [f\leq \mu]$ and $t_n\rightarrow +\infty$ such that $\frac{x_n}{t_n}\xrightarrow{\sigma} u$. Now for each $n\in \Bbb{N}$, we set $d_n=\frac{x_n}{t_n}$. Thus $f(t_n d_n)\leq\mu$ for each $n\in \Bbb{N}$. Therefore $\liminf_{n\to +\infty} f(t_n d_n)\leq\mu$ and we obtain $\alpha_u\leq\mu$. Then
	$$\alpha_u\leq f^{\sigma g}(u).$$
\end{proof}\qed
In the following proposition, we extend \cite[Remark 3.3]{H1} and \cite[Remark 3.1(i)]{H1}.
\begin{proposition}\label{P. property qr}
	Let $f:X\rightarrow \Bbb{R}\cup \{+\infty\}$ be a proper function, then
	\begin{itemize}
		\item[(i)] $\inf f^{\sigma g}=f^{\sigma g}(0)=\inf f.$
		\item[(ii)] $f^{\sigma g}$ is quasi-convex, $\sigma$-lsc (sequentially $\sigma$-lsc) and positively homogeneous of degree $0$, provided that $f$ be a quasi-convex and $\sigma$-lsc (sequentially $\sigma$-lsc) function.
	\end{itemize}
\end{proposition}
\begin{proof}
	$(i)$ By Proposition \ref{propo FFH}, it is clear that $f^{\sigma g}(0)\geq \inf f$. On the contrary, suppose that $f^{\sigma g}(0)> \inf f$. Therefore there exist $\lambda>\inf f$ and $x\in X$ such that $f^{\sigma g}(0)>\lambda> f(x)$. For each $n\in \Bbb{N}$ we set $t_n=n$ and $d_n=\frac{x}{n}$, so $f(x)=\liminf_{n\to \infty} f(t_n d_n)\geq f^{\sigma g}(0)>\lambda>f(x)$, which is a contradiction. Then, $f^{\sigma g}(0)=\inf f$. On the other hand, since by Proposition \ref{propo FFH} for each $u\in X$, $f^{\sigma g}(u)\geq \inf f$, then $\inf f^{\sigma g}=f^{\sigma g}(0)=\inf f$.\\
 $(ii)$ Let $f$ be a quasi-convex and $\sigma$-lsc function. Hence by (\ref{Eq1}) and Remark \ref{Rem1.cc}, for each $\lambda\in \Bbb{R}$, $[f^{\sigma g}\leq\lambda]$ is a $\sigma$-closed and convex cone. Then $f^{\sigma g}$ is quasi-convex, $\sigma$-lsc and positively homogeneous of degree $0$.\\
 If $f$ is sequentially $\sigma$-lsc we have the similar proof.
 \end{proof}
 In \cite{H1}, one of the most interesting results is \cite[Proposition 4.2]{H1}. Now the extension of this result, for infinite dimensional Banach spaces, is given below.
\begin{proposition}\label{boundedness of f and fqx}
Let $(X,\|.\|)$ be a Banach space and let $f:X\rightarrow \Bbb{R}\cup\{+\infty\}$ be a proper quasi-convex and sequentially $\sigma$-lsc function. Then $f^{\sigma g}$ is real-valued iff $f$ is a bounded function.
\end{proposition}
\begin{proof}
By Propositions \ref{propo FFH} and \ref{P. property qr}(i), we have
$$\inf f=f^{\sigma g}(0)=\inf f^{\sigma g}\leq \sup f^{\sigma g}\leq \sup f.$$
Then, $f^{\sigma g}$ is real-valued if $f$ is bounded. Also $f$ is bounded from below when $f^{\sigma g}$ is a real-valued function. Therefore we should prove that, if $f^{\sigma g}$ is a real-valued function, then $f$ is bounded from above. For this purpose by (\ref{Eq1}),
	\begin{equation}\label{7}
		X=\bigcup_{m=1}^\infty [f^{\sigma g}\leq m]=\bigcup_{m=1}^\infty([f\leq m])^\infty_{\sigma}.
	\end{equation}
Since $f$ is sequentially $\sigma$-lsc, for each $m\in\mathbb{N}$, $([f\leq m])^\infty_{\sigma}$ is sequentially $\sigma$-closed and so, $([f\leq m])^\infty_{\sigma}$ is closed in norm topology, for each $m\in\Bbb{N}$. Hence by Baire category theorem, there exists $m_0\in\Bbb{N}$ such that $([f\leq m_0])^\infty_{\sigma}$ has nonempty interior. Thus there exist $x_0\in X$ and a $\|.\|$-neighborhood $\mathcal{V}_{\|.\|}(x_0)$ of $x_0$ such that $\mathcal{V}_{\|.\|}(x_0)\subseteq ([f\leq m_0])^\infty_{\sigma}$ and there exists $m\in\Bbb{N}$ such that $-x_0\in ([f\leq m])^\infty_{\sigma}$. Then, there exists $p\in \Bbb{N}$ such that zero is in the interior of $([f\leq p])^\infty_{\sigma}$. Therefore, $([f\leq p])^\infty_{\sigma}=X$ and by Proposition \ref{Prop1}(iii),
$$[f\leq p]+([f\leq p])^\infty_{\sigma}\subseteq [f\leq p].$$
Then $[f\leq p]=X$ and $f$ is bounded from above with $p$ as an upper bound.
\end{proof}
By the following example we show that, in Proposition \ref{boundedness of f and fqx}, it is necessary that $f$ is a qusi-convex and sequentially $\sigma$-lsc function and we can not omit these conditions.
\begin{example}
Let $(X,\|.\|)=(l_p,\|.\|_p)$ for $1\leq p\leq\infty$, let $\sigma$ be the usual norm topology on $l_p$, and let $K=\{x=(x_i)\in l_p: x_i\in \mathbb{Q},~\forall i\in\mathbb{N}\}$. Suppose that $f(x)=\psi_K(x)$ (which is equal to $0$ on $K$ and equal to $+\infty$ on $X\setminus K$). Then $f^{\sigma g}\equiv 0$, but $f$ is not bounded from above, because, it is not quasi-convex nor sequentially $\sigma$-lsc.
\end{example}

\section{Existence Results}
In this section, as the main result of this paper, by using some asymptotic conditions, we obtain sufficient optimality conditions for existence of solutions to $(EP)$, when $K$ is an unbounded subset of $X$. For this purpose we need some definitions.\\
At first, motivated by the condition $(K)$ in \cite{MM} we have the following definition.
\begin{definition}
Let $\psi:K\times K\to \Bbb{R}$. We say that $\psi$ satisfies condition $(K_{\sigma})$, if for any unbounded sequence $(x_n)$ such that for each $n\in \Bbb{N}$, $x_n\in S(\psi,K_n)$, there exist $(x_{n_k})\subseteq (x_n)$, $(w_k)\subseteq K$ and $0 \neq d\in K^\infty_{\sigma}$ such that for each $k\in \Bbb{N}$, $\psi(x_{n_k},y)\leq \psi(w_k,y)$ for all $y\in K_{n_k}$, and $\frac{w_k}{\|w_k\|}\xrightarrow{\sigma} d$.
\end{definition}
\begin{example}\label{Exam.K-W}
\begin{itemize}
\item[(i)] Suppose that $(X,\|.\|)$ is a finite dimensional space, $K\subseteq X$ is nonempty, $\sigma=\|.\|$, and $\psi:K\times K\to \Bbb{R}$ is a bifunction. Then $\psi$ satisfies the condition $(K_{\sigma})$. Because, let $(x_n)$ be an unbounded sequence such that for each $n\in \Bbb{N}$, $x_n\in S(\psi,K_n)$, therefore there exist $(x_{n_k})\subseteq (x_n)$ and $0 \neq d\in K^\infty_{\sigma}$ such that $\frac{x_{n_k}}{\|x_{n_k}\|}\xrightarrow{\sigma} d$. It is sufficient that for each $k\in\Bbb N$ we set $w_k=x_{n_k}$.
\item[(ii)] Assume that $r\in [0,+\infty)$ and $\psi:K\times K\rightarrow\mathbb{R}$ is defined by $\psi(x,y)=r$ for each $x,y\in K$. Let $(x_n)\subseteq K$ is an unbounded sequence such that $x_n\in S(\psi, K_n)$ for each $n\in \mathbb{N}$. Let $x\in K\setminus \{0\}$. If for all $n\in \Bbb{N}$, we set $w_n=x$, therefore for each $n\in \Bbb{N}$, $\psi(x_{n},y)\leq \psi(w_n,y)$ for all $y\in K_{n}$, and $\frac{w_n}{\|w_n\|}\xrightarrow{\sigma} \frac{x}{\|x\|}\neq 0$. Then $\psi$ satisfies the condition $(K_{\sigma})$.
\item[(iii)] Let $\sigma=\|.\|$ and $f: K\to \Bbb R$ satisfies the condition $(K)$, that defined in \cite{MM}. Then it is not hard to see that the bifunction $\psi:K\times K\rightarrow\mathbb{R}$, that is defined by $\psi(x,y)=f(y)-f(x)$ for all $x,y\in K$, satisfies the condition $(K_{\sigma})$.
\end{itemize}
\end{example}
Now, motivated by definition of the set $R$ in \cite{FB}, related to $\psi:K\times K\rightarrow\mathbb{R}$, we consider the set
$$R_{\sigma}^{E}:=\{u\in K^\infty_{\sigma}: (-\psi(.,y))^{\sigma g}(u)\leq 0,~ \forall y\in K\}.$$
Now we have the main result of this paper.
\begin{theorem}\label{main thm}
Let $(X,\sigma)$ has the assumption $(H_\sigma)$ and $K$ be a nonempty $\sigma$-closed subset of $X$. Assume that $\psi: K\times K \to \Bbb{R}$ satisfies the condition $(K_{\sigma})$ and $S(\psi,K_n)$ is nonempty and $\sigma$-closed for each $n\in\Bbb{N}$. If $R_{\sigma}^E\subseteq\{0\}$, then $S(\psi,K)$ is nonempty and $(S(\psi,K))^\infty_{\sigma}=\{0\}$.
\end{theorem}
\begin{proof}
	Suppose that $x_n\in S(\psi, K_n)=\cap_{y\in K_n}\{x\in K_n: \psi(x,y)\geq 0\}$, for each $n\in\Bbb{N}$.
	Let $(x_n)$ has a $\sigma$-convergent subsequence $(x_{n_k})$ to $\bar{x}\in K$. Then there exists $m_0$ such that $\{x_{n_k}: k\in\Bbb{N}\}\cup \{\bar{x}\}\subseteq K_{m_0}$. Thus, for all $n_k\geq m\geq m_0$ and $y\in K_m$, $\psi(x_{n_k},y)\geq 0$ and $x_{n_k}\in K_m$. Hence for $n_k\geq m\geq m_0$, $x_{n_k}\in S(\psi,K_m)$. Therefore by $\sigma$-closedness of $S(\psi,K_m)$, for all $m\geq m_0$, $\bar{x}\in S(\psi,K_m)$. Then for each $y\in K$, $\psi(\bar{x},y)\geq 0$, and we have $\bar{x}\in S(\psi, K)$.\\
	Now assume that $\|x_n\|\rightarrow +\infty$. Then by the condition $(K_{\sigma})$, there exist $(x_{n_k})\subseteq (x_n)$, $(w_k)\subseteq K$ and $0\neq d\in K^\infty_{\sigma}$ such that for all $y\in K_{n_k}$, $\psi(x_{n_k},y)\leq\psi(w_k,y)$ and $\frac{w_k}{\|w_k\|}\xrightarrow{\sigma} d$.\\
If $(w_k)$ has a bounded subsequence, by the assumption $(H_{\sigma})$, it has a $\sigma$-convergent subsequence. Then by the similar proof, as the case that $(x_n)$ has a $\sigma$-convergent subsequence, $S(\psi, K)$ is nonempty.\\
	Therefore suppose that $\|w_k\|\rightarrow +\infty$. Since for all $y\in K_{n_k}$, $\psi(w_k,y)\geq \psi(x_{n_k},y)\geq 0$, hence for each $y\in K$, there exists $k_y\in\Bbb{N}$ such that for all $k\geq k_y$,
	$$-\psi(w_k,y)\leq -\psi(x_{n_k},y)\leq 0.$$
	Then
	$$(-\psi(.,y))^{\sigma g}(d)\leq 0.$$
	Which is a contradiction with $R_{\sigma}^E\subseteq\{0\}$.\\
	To show $(S(\psi,X))^\infty_{\sigma}=\{0\}$, let $u\in (S(\psi,K))^\infty_{\sigma}$. Hence there exists $(x_n)\subseteq	S(\psi,K)$ and $t_n\rightarrow +\infty$ such that $\frac{x_n}{t_n}\xrightarrow{\sigma} u$. Since for each $n\in\Bbb{N}$ and each $y\in K$, $-\psi(x_n,y)\leq 0$, therefore $(-\psi(.,y))^{\sigma g}(u)\leq 0$ for all $y\in K$. Thus $u\in R_{\sigma}^E\subseteq\{0\}$. Then $(S(\psi,K))^\infty_{\sigma}=\{0\}$.
\end{proof}
In the following example, we illustrate Theorem \ref{main thm}.
\begin{example}
	Suppose that $(X,\|.\|)=(l_p,\|.\|_p)$ for $1<p<\infty$, where $\|.\|_p$ is the standard norm on $l_p$, $\sigma$ is the weak topology on $X$, and $(\alpha_i)$ be a sequence in $\Bbb{R}$ such that $\sum_{i=1}^{+\infty}\alpha_i=+\infty$ and $\alpha_1\geq 1$. Put $C:=\{x=(x^i)\in \ell_p: 0\leq x^i\leq \alpha_{i}^{\frac{1}{p}}\}$. Then $C$ is a convex, $\sigma$-closed and unbounded subset of $\ell_p$ such that $C^{\infty}_{\sigma}=\{0\}$, \cite{M0}. Let $K=l_p$ and let $\psi:\ell_p\times\ell_p\rightarrow \Bbb{R}\cup\{+\infty\}$ be defined by
	$$\psi(x,y)=\left\{
	\begin{array}{ll}
		0,&~ x \vee y\in C,\\
		-1,&~ \mbox{otherwise}.
	\end{array}\right.$$
	Therefore for each $n\in \Bbb{N}$, $S(\psi,K_n)=\{x\in C: \|x\|_p\leq n\}$, which is nonempty and $\sigma$-closed. Also, let $(x_n)\subseteq K$ is an unbounded sequence such that $x_n\in S(\psi, K_n)$ for each $n\in \Bbb{N}$. For all $n\in \Bbb{N}$, we put $w_n=e_1$, where $e_1=(1,0,0,...)$, therefore for each $n\in \Bbb{N}$, $\psi(x_{n},y)\leq 0=\psi(w_n,y)$ for all $y\in K_n$ and $\frac{w_n}{\|w_n\|_p}\xrightarrow{\sigma} e_1\neq 0$. Hence $\psi$ satisfies the condition $(K_{\sigma})$. Also for each $u\in l_p$ and $y\notin C$ by Proposition \ref{propo FFH} we have
	$$(-\psi(.,y))^{\sigma g}(u)=\inf\{\liminf_{n\to +\infty} -\psi(t_nd_n,y): t_n\rightarrow +\infty, d_n\xrightarrow{\sigma} u\}=1>0.$$
Then $R_{\sigma}^E\subseteq\{0\}$, and all the conditions in Theorem \ref{main thm} are satisfied.\\
Note that in this example $S(\psi,K)=C$, thus $S(\psi,K)$ is nonempty and $(S(\psi,K))^\infty_{\sigma}=\{0\}$.
\end{example}
In the following corollary we have noncoercive version of \cite[Theorem 3.1]{NT} and \cite[Corollary 3.1]{NT} and some other results of Theorem \ref{main thm}.
\begin{corollary}\label{main cor.}
Suppose that $(X,\sigma)$ has the assumption $(H_\sigma)$, $B_X$ is $\sigma$-compact, and $K$ is a nonempty $\sigma$-closed and convex subset of $X$. Assume that $\psi: K\times K \to \Bbb{R}$ satisfies the condition $(K_{\sigma})$ and it is $\sigma$-tusc on $K_n$, for each $n\in\Bbb{N}$. Let $R_{\sigma}^E\subseteq\{0\}$. Then $S(\psi,K)$ is nonempty and $(S(\psi,K))^\infty_{\sigma}=\{0\}$, if at least one of the following holds.
	\begin{itemize}
		\item[(i)] For each $n\in \Bbb{N}$, $\psi|_{K_n\times K_n}(x, y)$ is anti-locally dominated in $y$.
		\item[(ii)] For each $n\in \Bbb{N}$, $\psi|_{K_n\times K_n}(x, y)$ is transfer quasi-convex in $y$.
		\item[(iii)] $\psi$ is cyclically anti-quasimonotone.
		\item[(iv)] For every $x\in K$, $\psi(x,.)$ is quasi-convex and for each $x\in K$, $\psi(x,x)\geq 0$.
		\item[(v)] $-\psi$ is pseudomonotone, for every $y\in K$, $\psi(.,y)$ is quasi-concave and for each $x\in K$, $\psi(x,x)=0$.
	\end{itemize}
\end{corollary}
\begin{proof}
Let $n\in\Bbb N$. By hypothesis $K_n$ is $\sigma$-compact. Since $\psi:K_n\times K_n\rightarrow \Bbb{R}$ is $\sigma$-tusc on $K$, then by the proof of \cite[Theorem 3.1]{NT} we can see that, $y\mapsto \{x\in K:\psi(x,y)\geq 0\}$ is $\sigma$-transfer closed on $K_n$. Therefore
	$$S(\psi,K_n)=\bigcap_{y\in K_n}\mathrm{cl}^\sigma (\{x\in K_n:\psi(x,y)\geq 0\}).$$
Hence for each $n\in \Bbb N$, $S(\psi,K_n)$ is $\sigma$-closed.\\
Therefore by Theorem \ref{main thm}, it is sufficient to show that in each case, for each $n\in \Bbb N$, $S(\psi,K_n)$ is a nonempty set.
\begin{itemize}
\item[(i)] By \cite[Theorem 3.1]{NT}, for each $n\in \Bbb N$, $S(\psi,K_n)$ is a nonempty set.
\item[(ii)] By \cite[Corollary 3.1]{NT}, for each $n\in \Bbb N$, $S(\psi,K_n)$ is a nonempty set.
\item[(iii)] By \cite[Proposition 1]{CS}, and Remark \ref{Remark c.a.qm}, for each $n\in \Bbb{N}$, $\psi|_{K_n\times K_n}(x, y)$ is anti-locally dominated in $y$. Then by $(i)$, for each $n\in \Bbb N$, $S(\psi,K_n)$ is a nonempty set.
\item[(iv)] 	For each $n\in \Bbb{N}$, $\psi|_{K_n\times K_n}(x, y)$ is transfer quasi-convex in $y$. Indeed, by the contrary suppose that, there exists $m\in\Bbb N$, such that $\psi|_{K_m\times K_m}(x, y)$ is not transfer quasi-convex in $y$. Therefore there exists finite subset $\{y_1,y_2,...,y_n\}\subseteq K_m$ and $x\in \mathrm{co}\{y_i: 1\leq i\leq n\}$ such that $\min_{i\in\{1,2,...,n\}}-\psi(x,y_i)>0$. Thus for each $i\in\{1,2,...,n\}$, $\psi(x,y_i)<0$. Then by quasi-convexity of $\psi(x,.)$, we have
$$\psi(x,x)\leq \max\{\psi(x,y_i): 1\leq i\leq n\}<0.$$
Which is a contradiction.\\
Then by $(ii)$, for each $n\in \Bbb N$, $S(\psi,K_n)$ is a nonempty set.\\
\item[(v)] Again, for each $n\in \Bbb{N}$, $\psi|_{K_n\times K_n}(x, y)$ is transfer quasi-convex in $y$. Because, by the contrary suppose that, there exists $m\in\Bbb N$, such that $\psi|_{K_m\times K_m}(x, y)$ is not transfer quasi-convex in $y$. Therefore, similarly to the proof of $(iv)$, there exists finite subset $\{y_1,y_2,...,y_n\}\subseteq K_m$ and $x\in \mathrm{co}\{y_i: 1\leq i\leq n\}$ such that for each $i\in\{1,2,...,n\}$, $\psi(x,y_i)<0$. Since $-\psi$ is pseudomonotone for each $i\in\{1,2,...,n\}$, we have $-\psi(y_i,x)<0$. Indeed, if there exists $i_0\in\{1,2,...,n\}$ such hat $-\psi(y_{i_0},x)\geq 0$, then we have $-\psi(x,y_{i_0})\leq 0$ and it is impossible. Now since $\psi(.,x)$ is quasi-concave,
$$-\psi(x,x)\leq \max\{-\psi(y_i,x): 1\leq i\leq n\}<0.$$
Which is a contradiction with $\psi(x,x)=0$.\\
Then by $(ii)$, for each $n\in \Bbb N$, $S(\psi,K_n)$ is a nonempty set.
\end{itemize}
\end{proof}
The following corollary is noncoercive version of Theorem \ref{main thm1} (\cite[Theorem 2.7]{KQ}).
\begin{corollary}
Suppose that $(X,\sigma)$ has the assumption $(H_\sigma)$, $B_X$ is $\sigma$-compact, and $K$ is a nonempty $\sigma$-closed subset of $X$. Assume that $\psi: K\times K \to \Bbb{R}$ satisfies the condition $(K_{\sigma})$ and for each $y\in K$, the set $[\psi(.,y)\geq 0]=\{x\in K:\psi(x,y)\geq 0\}$ is closed. Let $\psi$ be cyclically anti-quasimonotone and let $R_{\sigma}^E\subseteq\{0\}$. Then $S(\psi,K)$ is nonempty and $(S(\psi,K))^\infty_{\sigma}=\{0\}$.
\end{corollary}
\begin{proof}
Let $n\in\Bbb N$. By hypothesis $K_n$ is $\sigma$-compact. Then by \cite[Theorem 2.7]{KQ} for each $n\in \Bbb N$, $S(\psi,K_n)$ is a nonempty set. Also, since for each $y\in K_n$, the set $\{x\in K_n:\psi(x,y)\geq 0\}$ is $\sigma$-closed, then by
$$S(\psi,K_n)=\bigcap_{y\in K_n} \{x\in K_n:\psi(x,y)\geq 0\}=\bigcap_{y\in K_n}\mathrm{cl}^\sigma (\{x\in K_n:\psi(x,y)\geq 0\}),$$
$S(\psi,K_n)$ is $\sigma$-closed. Hence for each $n\in \Bbb N$, $S(\psi,K_n)$ is nonempty and $\sigma$-closed. Then, by Theorem \ref{main thm}, $S(\psi,K)$ is nonempty and $(S(\psi,K))^\infty_{\sigma}=\{0\}$.
\end{proof}
Related to the previous results, this problem naturally arises: When do we have $R_{\sigma}^E\subseteq\{0\}$?\\
We answer to this question in two following propositions.
\begin{proposition}\label{Prop.Replace.condition}
Let $K$ be a nonempty subset $X$ and let $\psi: K\times K \to \Bbb{R}$. Assume that for each $u\in K^\infty_{\sigma}\setminus \{0\}$ there exists $y\in K$ such that $(-\psi(.,y))^{\infty}_{\sigma}(u)>0$, then $R_{\sigma}^E\subseteq\{0\}$.
\end{proposition}
\begin{proof}
Let $u\in K^\infty_{\sigma}\setminus \{0\}$. Hence there exists $y\in K$ such that $(-\psi(.,y))^{\infty}_{\sigma}(u)>0$. Then, there exists $\alpha>0$ such that for all $t_n\rightarrow +\infty$ and $d_n\xrightarrow{\sigma} u$,
	$$\liminf_{n\rightarrow+\infty} \frac{-\psi(t_nd_n,y)}{t_n}>\alpha.$$
	Therefore for large enough $n\in\Bbb{N}$,
	$$-\psi(t_nd_n,y)\geq \alpha t_n>\alpha.$$
	Thus
	$$\liminf_{n\rightarrow+\infty}-\psi(t_nd_n,y)\geq\alpha.$$
	Hence by Proposition \ref{propo FFH}, we conclude that $(-\psi(.,y))^{\sigma g}(u)>0$. Then $R_{\sigma}^E\subseteq\{0\}$.
\end{proof}\qed
\begin{remark}\label{Remark.weak asymptotic}
The converse of Proposition \ref{Prop.Replace.condition} is not necessarily true. Because, let $\psi: \Bbb{R}\times \Bbb{R}\to \Bbb{R}$ is defined by $\psi(x,y)=y$ for each $x,y\in\Bbb{R}$. Then for every $u\in \Bbb R$ and $y\in (-\infty,0)$ we have $(-\psi(.,y))^{\sigma g}(u)=y>0$, so $R_{\sigma}^E\subseteq\{0\}$. But, for $u\neq 0$ and $y\in\Bbb{R}$, $(-\psi(.,y))^{\infty}_{\sigma}(u)=0$.
\end{remark}
Also, we have the following proposition.
\begin{proposition}\label{Prop.converse.main thm}
Let $K$ be a nonempty convex subset of $X$ and let $\psi: K\times K \to \Bbb{R}$. Assume that for each $y\in K$, $\psi(.,y)$ is quasi-concave and $(n,y)\mapsto [-\psi(.,y)\leq \frac{1}{n}]$ is $\sigma$-intersectionally closed on $\Bbb{N}\times K$. If $S(\psi,K)$ is a nonempty $\sigma$-closed set and $(S(\psi,K))^\infty_{\sigma}=\{0\}$, then $R_{\sigma}^E\subseteq\{0\}$.
\end{proposition}
\begin{proof}
Let $u\in R_{\sigma}^E$. Then by the definition of $\sigma g$-asymptotic function, for each $n\in\Bbb{N}$ and $y\in K$, $u\in ([-\psi(.,y)\leq \frac{1}{n}])^{\infty}_{\sigma}$. Thus
\begin{equation}\label{Eq2}
u\in \bigcap_{(n,y)\in \Bbb{N}\times K}([-\psi(.,y)\leq \frac{1}{n}])^{\infty}_{\sigma}.
\end{equation}
Since $S(\psi,K)$ is a nonempty $\sigma$-closed set, $S(\psi,K)=\cap_{(n,y)\in \Bbb{N}\times K}[-\psi(.,y)\leq \frac{1}{n}]$, $(n,y)\mapsto [-\psi(.,y)\leq \frac{1}{n}]$ is $\sigma$-intersectionally closed on $\Bbb{N}\times K$ and for each $(n,y)\in \Bbb{N}\times K$, $[-\psi(.,y)\leq \frac{1}{n}]$ is convex, therefore by (\ref{Eq2}), Proposition \ref{Prop1}(i) and Proposition \ref{Prop1}(iv) we have
\begin{eqnarray*}
		u&\in& \bigcap_{(n,y)\in \Bbb{N}\times K}([-\psi(.,y)\leq \frac{1}{n}])^{\infty}_{\sigma}\\
		&&\subseteq \bigcap_{(n,y)\in \Bbb{N}\times K}(\mathrm{cl}^\sigma([-\psi(.,y)\leq \frac{1}{n}]))^{\infty}_{\sigma}\\
		&&= (\bigcap_{(n,y)\in \Bbb{N}\times K}\mathrm{cl}^\sigma([-\psi(.,y)\leq \frac{1}{n}]))^{\infty}_{\sigma}\\
		&&=(\mathrm{cl}^\sigma(\bigcap_{(n,y)\in \Bbb{N}\times K}[-\psi(.,y)\leq \frac{1}{n}]))^{\infty}_{\sigma}\\
		&&=(\mathrm{cl}^\sigma(S(\psi,K)))^{\infty}_{\sigma}
		=(S(\psi,K))^{\infty}_{\sigma}=\{0\}.
\end{eqnarray*}
Then $R_{\sigma}^E\subseteq\{0\}$.
\end{proof}\qed
Now by using Corollary \ref{main cor.} and Proposition \ref{Prop.converse.main thm} we obtain a necessary and sufficient optimality conditions for existence of solutions to $(EP)$, when $K$ is an unbounded subset of $X$.
\begin{corollary}
Suppose that $(X,\sigma)$ has the assumption $(H_\sigma)$, $B_X$ is $\sigma$-compact, and $K$ is a nonempty $\sigma$-closed and convex subset of $X$. Let $\psi: K\times K \to \Bbb{R}$ satisfies the condition $(K_{\sigma})$, and $\psi(x,x)=0$, for each $x\in K$. Also, assume that $-\psi$ is pseudomonotone and for each $y\in K$, $\psi(.,y)$ is a quasi-concave and $\sigma$-usc function. Then $R_{\sigma}^E\subseteq\{0\}$, iff $S(\psi,K)$ is nonempty and $(S(\psi,K))^\infty_{\sigma}=\{0\}$.
\end{corollary}
\begin{proof}
Since for each $y\in K$, $\psi(.,y)$ is a $\sigma$-usc function, thus for each $y\in K$, the set $\{x\in K: \psi(x,y)\geq 0\}$ is $\sigma$-closed, and moreover, for each $(n,y)\in \Bbb{N}\times K$, $[-\psi(.,y)\leq \frac{1}{n}]$ is $\sigma$-closed. Therefore by $S(\psi, K)=\cap_{y\in K}\{x\in K: \psi(x,y)\geq 0\}$, $S(\psi, K)$ is a $\sigma$-closed set. Also, $(n,y)\mapsto [-\psi(.,y)\leq \frac{1}{n}]$ is $\sigma$-intersectionally closed on $\Bbb{N}\times K$. Then, this corollary is a consequence of Corollary \ref{main cor.} and Proposition \ref{Prop.converse.main thm}.
\end{proof}

\section{Application in Optimization}
In this section, as an application of results in the preceding sections, we establish a result for existence solutions to minimization problem.

Let $f: K\to\Bbb{R}$ be a function. We consider the following minimization problem:
$$(MP)\quad \min f(x) \quad \mbox{subject to} \quad x\in K.$$
The set of its solutions is denoted by $argmin(f,K)$. Let $\psi: K\times K\to\Bbb{R}$ is defined by $\psi(x,y)=f(y)-f(x)$ for all $x,y\in K$. Then it is clear that $\bar{x}\in argmin(f,K)$ iff, $\bar{x}\in S(\psi,K)$. Now by Theorem \ref{main thm} we have the following result, which is extended the result in \cite[Theorem 4.1]{H1}, to infinite dimensional spaces under weaker assumption of continuity and without quasi-convexity assumption.
\begin{theorem}\label{main theorem in minimization}	
Suppose that $(X,\sigma)$ has the assumption $(H_\sigma)$, $B_X$ is $\sigma$-compact, and $K$ is a nonempty $\sigma$-closed and convex subset of $X$. Let $f: K\to\Bbb{R}$ be bounded from below and let for each $n\in\Bbb{N}$, $f|_{K_n}$ be $\sigma$-transfer lower continuous. Suppose that for every $u\in K^\infty_{\sigma}\setminus \{0\}$, $f^{\sigma g}(u)>f^{\sigma g}(0)$. Then $(MP)$ has a solution and $(argmin(f,K))^\infty_{\sigma}=\{0\}$, if the following condition holds.
	\begin{itemize}
\item[$(K_{\sigma}^m)$] If there exists an unbounded sequence $(x_n)$ such that for each $n\in \Bbb{N}$, $x_n\in argmin(f,K_n)$, then there exist $(n_k)\subseteq (n)$, $(w_k)\subseteq K$ and $0 \neq d\in K^\infty_{\sigma}$ such that for each $k\in \Bbb{N}$, $w_k\in argmin(f,K_{n_k})$ and $\frac{w_k}{\|w_k\|}\xrightarrow{\sigma} d$.
	\end{itemize}
\end{theorem}
\begin{proof}
Let $\psi: K\times K\to\Bbb{R}$ is defined by $\psi(x,y)=f(y)-f(x)$ for all $x,y\in K$. By \cite[Theorem 2]{T1} for each $n\in\Bbb{N}$, $argmin(f,K_n)=S(\psi,K_n)$ is nonempty and $\sigma$-compact. Also it is not hard to see that the condition $(K_{\sigma}^m)$ implies that, $\psi$ satisfies the condition $(K_{\sigma})$.\\
Let $u\in K^\infty_{\sigma}\setminus \{0\}$, therefore by hypothesis and Proposition \ref{P. property qr}(i) we have, $f^{\sigma g}(u)>f^{\sigma g}(0)=\inf_K f$. Thus there exists $y\in K$ such that $f^{\sigma g}(u)>f(y)$. Hence by Proposition \ref{propo FFH}, $(-\psi(.,y))^{\sigma g}(u)>0$. Therefore $R_{\sigma}^E\subseteq\{0\}$. Then by $argmin(f,K)=S(\psi,K)$ and Theorem \ref{main thm}, $argmin(f,K)$ is nonempty and $(argmin(f,K))^\infty_{\sigma}=\{0\}$.
\end{proof}
Note that, the condition $(K_{\sigma}^m)$ is an extension of condition $(K)$, that introduced in \cite{MM}, to infinite dimensional spaces.\\
In the following example, we illustrate Theorem \ref{main theorem in minimization}.
\begin{example}
Let $(X,\|.\|)=(l_2,\|.\|_2)$ and let $\sigma$ be the weak topology. Suppose that  $f:l_2\rightarrow \Bbb R$  is defined by
	$$f(x)=\left\{
	\begin{array}{ll}
	\tan^{-1}\|x\|_2, \quad\quad\quad\quad & \|x\|_2\geq 1,\\
	\frac{1}{5}\|x\|_2, & \frac{1}{2}<\|x\|_2<1,\\
	-1,&\|x\|_2\leq \frac{1}{2}.
	\end{array}\right.$$
It's not hard to see that, for each $n\in\Bbb N$, $f$ is $\sigma$-transfer lower continuous on $K_n$ and $f$ satisfies the condition $(K_\sigma^m)$.

Since
	$$[f\leq \lambda]=\left\{
	\begin{array}{ll}
	l_2, \quad\quad\quad\quad & \lambda\geq\frac{\pi}{2},\\
	\{x\in l_2: \|x\|_2\leq \tan \lambda\}, &~ \frac{\pi}{4}\leq \lambda<\frac{\pi}{2},\\
	\{x\in l_2:\|x\|_2<1\},&\frac{1}{5}<\lambda<\frac{\pi}{4},\\
	\{x\in l_2:\|x\|_2\leq 5\lambda\},&\frac{1}{10}<\lambda<\frac{1}{5},\\
	\{x\in l_2:\|x\|_2\leq \frac{1}{2}\},&-1\leq \lambda<\frac{1}{10},\\
	\emptyset,&\lambda<-1,
	\end{array}\right.$$
then,
	$$({[f\leq \lambda]})^\infty_{\sigma}=\left\{
	\begin{array}{ll}
	l_2, \quad\quad\quad\quad & \lambda\geq\frac{\pi}{2},\\
	\{0\}, & -1\leq \lambda<\frac{\pi}{2},\\
	\emptyset,&\lambda<-1.
	\end{array}\right.$$
	Therefore,
	$$f^{\sigma g}(u)=\left\{
	\begin{array}{ll}
	\frac{\pi}{2}, \quad\quad\quad\quad & u\neq 0,\\
	-1, & u=0.
	\end{array}\right.$$
	Hence, $f^{\sigma g}(u)>\inf f$ for all $u\neq 0$. Then, all the conditions in Theorem \ref{main theorem in minimization} are satisfied.
\end{example}

By the following example we see that, Theorem \ref{main theorem in minimization} extend sufficient conditions for existence of solutions to
minimization problems in \cite[Theorem 4.1]{H1}, even in finite dimensional spaces.
\begin{example}
Let $X=\Bbb R$ and let $\sigma$ be the usual topology on $\Bbb R$.	Suppose that $f:\Bbb R\rightarrow \Bbb R$ is defined by
	$$f(x)=\left\{
	\begin{array}{ll}
		-tan^{-1}x,&~~~~~~~~~~ x\leq 0,\\
		x, &~~~~~~~~~~ 0<x\leq \frac{\sqrt{3}}{6},\\
		-x+ \frac{\sqrt{3}}{3}, & ~~~~~~~~~~\frac{\sqrt{3}}{6}\leq x< \frac{\sqrt{3}}{3},\\
		tan^{-1}x,&~~~~~~~~~~ x\geq \frac{\sqrt{3}}{3}.
	\end{array}\right.$$
	Therefore $f^{\sigma g}(u)=\frac{\pi}{2}>f^{\sigma g}(0)=\inf f=0$, for all $u\neq 0.$ Also, $f$ is $\sigma$-transfer lower continuous on $K_n$, for all $n\in \mathbb{N}$. Then by Theorem \ref{main theorem in minimization}, $argmin(f,\mathbb{R})$ is a nonempty $\sigma$-compact set, but $f$ is not $\sigma$-lsc nor quasi-convex.
\end{example}

\begin{acknowledgements}
The second author was partially supported by a grant from IPM (No. 1401460421).
\end{acknowledgements}


%
%



\end{document}